\documentclass[reqno]{amsart}
\usepackage{amssymb,amsmath}
\usepackage[all]{xy}
\usepackage{graphicx}

\numberwithin{equation}{section}
\theoremstyle{theorem}
\newtheorem{theorem}{Theorem}[section]
\newtheorem{corollary}[theorem]{Corollary}
\newtheorem{lemma}[theorem]{Lemma}
\newtheorem{prop}[theorem]{Proposition}

\theoremstyle{definition}
\newtheorem{definition}{Definition}[section]
\newtheorem{remark}[definition]{Remark}

\newcommand{\AI}{A_\infty}

\newcommand{\bx}{\boldsymbol{x}}
\newcommand{\ue}{\underline{e_i}}
\newcommand{\ux}{\underline{\boldsymbol{x}}}

\newcommand{\kk}{\boldsymbol{k}}
\newcommand{\ZZ}{\mathbb{Z}}

\newcommand{\NN}{\mathbb{N}}

\newcommand{\by}{\boldsymbol{y}}

\newcommand{\WH}[1]{\widehat{#1}}
\newcommand{\UL}[1]{\underline{#1}}

\newcommand{\WT}[1]{\widetilde{#1}}

\newcommand{\tdy}{h(\boldsymbol{y})}

\begin{document}

\title{Potentials of homotopy cyclic $\AI$-algebras}
\author{Cheol-Hyun Cho}
\email{chocheol@snu.ac.kr}
\address{Department of Mathematical Sciences, 
         Research Institute of Mathematics,
         Seoul National University,
         Gwanak-gu, Seoul, 151-747
         South Korea}
\author{Sangwook Lee}
\email{leemky7@snu.ac.kr}
\address{Department of Mathematical Sciences, 
         Seoul National University,
         Gwanak-gu, Seoul, 151-747
         South Korea}
\thanks{This work was supported by the Korea Research Foundation Grant funded by the Korean Government (MOEHRD, Basic Research Promotion Fund) (KRF-2008-C00031)}
\begin{abstract}
For a cyclic $\AI$-algebra, a potential recording the structure constants can be defined. We define an analogous potential for a homotopy cyclic  $\AI$-algebra and prove its properties. 
On the other hand, we find another different potential for a homotopy cyclic  $\AI$-algebra, which is related to
the algebraic analogue of generalized holonomy map of Abbaspour, Tradler and Zeinalian. 
\end{abstract}
\maketitle

\section{Introduction}
We first recall the definition of cyclic inner products due to Kontsevich \cite{K}, which may be understood as constant invariant symplectic structures in the non-commutative geometry.
\begin{definition}\label{def:cyc}
An $\AI$-algebra $(A,\{m_*\})$ is said to have a \textbf{cyclic inner product} if
there exists a skew symmetric non-degenerate, bilinear map $$<,> : A \otimes A \to \kk$$
such that for all integer $k \geq 1$,
\begin{equation}\label{cycsym}
	<m_k(x_1,\cdots,x_k),x_{k+1}> = (-1)^{K(\vec{x})}<m_k(x_2,\cdots,x_{k+1}),x_{1}>.
\end{equation}
Here, $(-1)^{K(\vec{x})}$ denotes the sign given by Koszul sign convention.
Namely,
\begin{equation}
(-1)^{K(\vec{x})} = (-1)^{|x_1|'(|x_2|' + \cdots +|x_{k+1}|')},
\end{equation}
where $|x|'$ is the shifted degree of $x$.
\end{definition}
This notion for the $\AI$-algebras and $\AI$-categories is crucial in homological mirror symmetry, for example,
 as in the work of Kontsevich-Soibelman\cite{KS} or  Costello\cite{Cos}. In particular, Costello has proved in \cite{Cos}
that the category of open topological conformal field theory is homotopy equivalent to the category of Calabi-Yau categories,
where the Calabi-Yau category is a categorical generalization of a cyclic $\AI$-algebra.

The first application of this gadget is to define a potential for a cyclic $\AI$-algebra, which in physics, is called an action of a string field theory: Let $(A,m^A_*)$ be a cyclic $\AI$-algebra.
Let $e_i$ be generators of $A$ as a vector space, which is assumed to be finite dimensional.
Define $\bx = \sum_i e_i x_i$ where $x_i$ are formal parameters with $deg(x_i) = - deg(e_i)$.
\begin{definition}\label{podef}
Define
\begin{equation}
\Phi^A(\bx) = \sum_{k=1}^\infty  \frac{1}{k+1} < m^A_k( \bx,\bx,\cdots,\bx),\bx>	
\end{equation}
\end{definition}
This may be considered as a systematic way of gathering structure constants of a cyclic $\AI$-algebra.  In the case of toric manifolds, this potential when
restricted to the Maurer-Cartan elements, becomes the Landau-Ginzburg superpotential of the mirror B-model (see \cite{CO},\cite{FOOO1}).

The notion of cyclicity is not  a homotopy invariant notion.  For example, an $\AI$-algebra which is homotopy equivalent to a cyclic $\AI$-algebra may not be cyclic. Instead, it has a  strong homotopy inner product, which was
defined by the first author in  \cite{Cho}:
for this, a cyclic inner product on $A$ may be understood as a special kind of $\AI$-bimodule map $A \to A^*$ (Lemma 3.1, \cite{Cho}).
An $\AI$-bimodule quasi-isomorphism $A\to A^*$ is called as an infinity inner product by Tradler (see \cite{T},\cite{TZ} for example).
\begin{definition}[\cite{Cho}, Definition 3.6]\label{shidef}
Let $A$ be an $\AI$-algebra.
We call an $\AI$-bimodule map $\phi:A \to A^*$ a \textbf{strong homotopy inner product} if
there exists a cyclic $\AI$-algebra $B$ with $\psi:B \to B^*$ and
an $\AI$-quasi-isomorphism $f:A \to B$ such that the following diagram of $\AI$-bimodules over $A$ commutes
\begin{equation}\label{comm3}
\xymatrix@C+1cm{ A \ar[r]_{ g = \widetilde{f}}  \ar[d]_\phi & \, B \ar[d]_\psi^{cyc} \\ A^*  &  \ar[l]^{g^*} B^*}	
\end{equation}
Here by $g:A \to B$, we denote the induced $\AI$-bimodule map $\widetilde{f}=g$ where $B$ is considered as an $\AI$-bimodule over $A$.
\end{definition}

In this paper, we give a definition of the potential for strong homotopy inner products and  prove its properties. It turns out that the definition of the potential \ref{def:homotopypotential} is very similar to
that of (\ref{podef}), but the prove that they are indeed related  is non-trivial and involves quite  combinatorial arguments.  Beyond the fact that it is quite natural to work with homotopy notions when
dealing with homotopy algebras, sometimes it is necessary to work with homotopy notions directly.
For example, in the work of Kontsevich and Soibelman \cite{KS}, they find a relation between cyclic cohomology of an $\AI$-algebra $A$ and cyclic symmetry. Given a cyclic cohomology class, one obtains first a homotopy inner product on $A$, and then cyclic inner product in the minimal model. We refer readers to \cite{CL} for
the explicit formulas of this correspondence in terms of negative cyclic cohomology $HC_{-}^\bullet(A)$ and strong homotopy inner products.

Now, let us assume that the $\AI$-algebra is unital(see definition \ref{def:unit}), and assume that the $\AI$-bimodule maps are also unital. Then, from the strong homotopy inner products $\{<,>_{p,q}\}$ we can define another potential as
follows. (Here, $<,>_{p,q}$ is obtained from the $(p,q)$-component of the bimodule map $\phi$. See (\ref{eq:pq}))
\begin{definition}\label{podef2}
Define
\begin{equation}
\Psi^A(\bx) = \sum_{p,q\geq 0}  \frac{1}{p+q+1} <  \underbrace{\bx,\cdots, \bx}_{p},\UL{\bx},\underbrace{\bx,\cdots,\bx}_{q}| I >_{p,q}
\end{equation}

\end{definition}
We prove that this potential is in fact {\em invariant under the gauge equivalence} for Maurer-Cartan elements. We also find its relation to the work of
Abbaspour, Tradler and Zeinalian \cite{ATZ}, where this map corresponds to the algebraic analogue of the generalized holonomy map from the negative cyclic cohomology to the function ring of Maurer-Cartan elements.

\begin{figure}
\begin{center}
\includegraphics[height=1.5in]{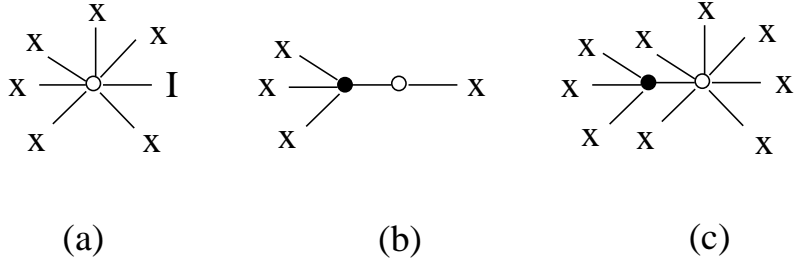}
\caption{(a) Potential $\Psi$, (b) Cyclic Potential $\Phi$, (c) Homotopy cyclic potential $\Phi$}
\label{potentials}
\end{center}
\end{figure}
The following Figure \ref{potentials} explains the differences of the expressions used in these potentials (without the coefficients). In the figure, the circle represent the strong homotopy inner product (following that of Tradler \cite{T}) whose horizontal arrows are for the inputs from modules. The filled circle represent the $\AI$-operation $m$.

This paper may be considered as a continuation of the paper \cite{Cho} to which
we refer readers for the notations and further introductions, especially about the signs. Throughout the paper
we assume that $H^\bullet(A)$ is finite dimensional. We thank H. Kajiura for sending us the unpublished manuscript (with Y. Terashima) on the decomposition theorem of $\AI$-algebras.

\section{Strong homotopy inner products}
We begin by proposing a modified definition of strong homotopy inner products, and discuss their equivalences and pull-backs.

We first make an observation that there exists certain subtlety in the direction of arrows in the diagram \ref{comm3} in the definition of strong
homotopy inner products. One could try to make the definition with the arrow $A \longleftarrow B$ instead of
$A \longrightarrow B$, but the resulting diagram would become weaker as there may exists elements of $A$ which is not covered by the image of the map $A \longleftarrow B$ in general. The subtlety actually disappears if we have non-degeneracy in the chain level.  The correct definition (which corresponds to exactly
the non-commutative invariant symplectic two form) is rather in between these two definitions: to make
the correct definition, we first recall the following characterization theorem of
strong homotopy inner products from \cite{Cho}.
\begin{theorem}[\cite{Cho}, Theorem 5.1]\label{thm:shi}
An $\AI$-algebra $A$ has a strong homotopy inner product if and only if
there exists an $\AI$-bimodule map $\phi:A \to A^*$, satisfying the following three conditions.
\begin{enumerate}
\item (Skew symmetry) For any $a_i,v,b_j,w \in A$,
$$\phi_{k,l}(\vec{a},v, \vec{b})(w) = -(-1)^{K}\phi_{l,k}(\vec{b},w,\vec{a})(v),$$
with $|K| = ( \sum_{i=1}^k |a_i|' + |v|')( \sum_{j=1}^l |b_j|' + |w|')$
\item (Closedness)
For any choice of a family $(a_1,\cdots,a_{l+1})$ and any choice of indices $1 \leq i<j<k\leq l+1$,
we have
$$(-1)^{K_i} \phi(.. ,\UL{a_i},..)(a_j) + (-1)^{K_j} \phi(.. ,\UL{a_j},..)(a_k) +
(-1)^{K_k} \phi(.. ,\UL{a_k},..)(a_i) =0,$$
where the arguments inside $\phi$ are uniquely given by the cyclic order of the family $(a_1,\cdots,a_{l+1})$,
and the signs $K_*$ are given by the Koszul convention:
\begin{equation}\label{eq}
K_* = (|a_1|' + \cdots + |a_*|')(|a_{*+1}|'+\cdots + |a_k|').
\end{equation}
\item (Homological non-degeneracy) For any non-zero $[a] \in H^\bullet(A)$ with $a \in A$, there exists a $[b] \in H^\bullet(A)$ with $b \in A$,  such that
$\phi_{0,0}(a)(b) \neq 0$.
\end{enumerate}
For non-degeneracy on the chain level, $\phi$ itself gives the strong homotopy inner product, otherwise the inner product obtained $\phi':A \to A^*$
is only equivalent to $\phi$.
\end{theorem}
The second condition is called {\em closed condition} since it is equivalent to the closed condition of the related non-commutative symplectic 2-form, and this plays a crucial role in proving the properties of the potential defined in this paper.

We also remark that in the proof of the Theorem \ref{thm:shi}, $\phi$ satisfying the three conditions, does not always become exactly a strong homotopy inner product (in the sense of definition \ref{shidef}),
but only equivalent  to  a strong homotopy inner product (the equivalence is defined below).

Hence, we propose to define the strong homotopy inner products by the Theorem \ref{thm:shi} because such a definition is equivalent to that of the non-commutative symplectic form as explained in \cite{Cho}.
\begin{definition}\label{def:cyc2}
    Let $A$ be an $A_{\infty}$-algebra. We call an $A_{\infty}$-bimodule map $\phi:A \rightarrow A^*$ a
    {\textbf{strong homotopy inner product}} if it is
    skew-symmetric, closed and homologically non-degenerate as in Theorem \ref{thm:shi}.

    And $A$ is called \textbf{homotopy cyclic $\AI$-algebra}, if there exists a strong homotopy inner product of $A$.
\end{definition}

Then, the main result of \cite{Cho} can be phrased as  the following theorem.
\begin{theorem}\label{prop:shi}
 Let $\phi:A \rightarrow A^*$  be an $A_{\infty}$-bimodule map.
\begin{enumerate}
\item If $\phi$ is a strong homotopy inner product in the sense of definition \ref{def:cyc2},
then there exists an $\AI$-algebra $B$ with a cyclic inner product
    $\psi: B \rightarrow B^*$
        and an $\AI$-quasi-isomorphism $\iota:B \rightarrow A$ satisfying the
    following commutative diagram of $A_{\infty}$-bimodule homomorphisms
\begin{equation}\label{defdiagram1}
\xymatrix{A \ar[d]_{\phi} & B \ar[l]_{\tilde{\iota}} \ar[d]^{cyc}_{\psi} \\ A^* \ar[r]^{{\tilde{\iota}}^*} & B^* }
\end{equation}

\item If there exists a cyclic $\AI$-algebra $B$ with $\psi:B \to B^*$ and
an $\AI$-quasi-isomorphism $f:A \to B$ such that the following diagram of $\AI$-bimodules over $A$ commutes
\begin{equation}
\xymatrix@C+1cm{ A \ar[r]_{ g = \widetilde{f}}  \ar[d]_\phi & \, B \ar[d]_\psi^{cyc} \\ A^*  &  \ar[l]^{g^*} B^*}	
\end{equation}
then, $\phi$ is a strong homotopy inner product in the sense of definition \ref{def:cyc2}.
\end{enumerate}
If $\phi_{0,0}$ is non-degenerate in the chain level, the old and the new definitions of a strong homotopy inner product are
equivalent.
\end{theorem}

\begin{remark}
Hence the new definition of the strong homotopy inner product is a little stronger than the diagram using
$A \longleftarrow B$, a little weaker than the diagram using
$A \longrightarrow B$ and equivalent to the non-commutative symplectic two form.
\end{remark}
\begin{proof}
In the non-degenerate case, from the proof of Theorem \ref{thm:shi}, one can find $B$ with an $\AI$-isomorphism $f:A \to B$ with the commuting diagram (\ref{comm3}). Hence one can find exact inverse of $f$ to make the commuting diagram (\ref{defdiagram1}).

Also, the statement $(2)$ can be checked without much difficulty from the commuting diagram, so we only consider the
statement $(1)$.
We explain that the proof of the theorem \ref{thm:shi} given in \cite{Cho} is enough to prove the existence of the diagram (\ref{defdiagram1}):
We recall from \cite{Cho} that the first step of the construction of cyclic $\AI$-algebra $B$ when $A$ is only
homologically non-degenerate was considering the minimal model $\iota : H^\bullet(A) \to A$ and consider the pull back $\iota^* \phi$.
\begin{equation}\label{defdiagram2}
\xymatrix{A \ar[d]_{\phi} & H^\bullet(A) \ar[l]_{\widetilde{\iota}} \ar[d]_{\iota^* \phi} \ar[r]_{\widetilde{f}} & H^\bullet(A) \ar[d]_{cyc} \\ A^* \ar[r]^{\widetilde{\iota}^*} & (H^\bullet(A))^*
& \ar[l]_{\widetilde{f}^*} (H^\bullet(A))^*}
\end{equation}
Then $\iota^* \phi$ is non-degenerate and skew symmetric and closed, and one proves the theorem for $\iota^* \phi$ to find
$f:H^\bullet(A) \to H^\bullet(A)$ with the above commutative
diagram. As the quasi-isomorphism $f$ on $H^\bullet(A)$ is in fact  an isomorphism, hence  there exists explicit inverse $f^{-1}$ and we obtain the diagram (\ref{defdiagram1}).
\end{proof}
We can also prove the following corollary.
\begin{corollary}
Let $\phi:A \rightarrow A^*$ be a strong homotopy inner product. Suppose we have an $\AI$-quasimorphism $f:A \to H^\bullet(A)$
with the commuting diagram
\begin{equation}\label{comm5}
\xymatrix@C+1cm{ A \ar[r]_{ g = \widetilde{f}}  \ar[d]_\phi & \, H^\bullet(A) \ar[d]_\psi^{cyc} \\ A^*  &  \ar[l]^{\widetilde{g}^*} H^\bullet(A)^*}	
\end{equation}
then, there exists an $\AI$-quasimorphism $h: H^*(A) \to A$ with the commuting diagram
(with the same $\psi$ as the above)
\begin{equation}\label{comm6}
\xymatrix@C+1cm{A  \ar[d]_\phi &  \ar[l]_{\tilde{h}} \, H^\bullet(A) \ar[d]_\psi^{cyc} \\ A^*  \ar[r]^{\tilde{h}^*} &   (H^\bullet(A))^*}	
\end{equation}
\end{corollary}
\begin{proof}
By the decomposition theorem of $\AI$-algebras, the map $f$ has a right inverse $\AI$-quasi-homomorphism, say $h:  H^*(A) \to A$ such that $f \circ h = id$.  To see this, consider an $\AI$-isomorphism $\eta$
$$\eta: A \to A^{dc} := A^H \oplus A^{lc}$$ to the direct
sum of the minimal $\AI$-algebra $A^H$ and the linear contractible $A^{lc}$.

 Let $\pi :A^{dc} \to A^H$ be the
projection and $i:A^H \to A^{dc}$ be the inclusion where the both are $\AI$-quasimorphisms with $ \pi \circ i = id$.
As $f$ is an $\AI$-quasimorphism, $f \circ \eta^{-1} \circ i:A^H \to H^\bullet(A)$ is an $\AI$-isomorphism, hence has an $\AI$-inverse say $\xi$.
Then, we define the right $\AI$ inverse $h = \eta^{-1} \circ i \circ \xi$. The property $ f \circ h = id$ can be checked immediately.
The second diagram then follows from the first commuting diagram.
\end{proof}

Now, we define equivalences between strong homotopy inner products.
\begin{definition}\label{def:equiv}
  Let $\phi: A \rightarrow A^*$ and $\psi: B \rightarrow B^*$ be strong homotopy inner products. They are called {\textbf{equivalent}} if there exists a cyclic symmetric $\AI$-algebra $H$ with a commutative diagram:
     $$\xymatrix{A \ar[d]_{\phi} & H \ar[d]^{cyc} \ar[l]_{qis} \ar[r]^{qis} & B \ar[d]^{\psi} \\
            A^* \ar [r] & H^* & B^* \ar[l]}$$
\end{definition}
One can actually choose $H$ to be a minimal(or canonical) model.

Given a strong homotopy inner product $\phi: B \to B^*$, and an $\AI$-quasi-isomorphism $f: A \to B$, we may define a pullback $f^* \phi : A \to A^*$ as a composition : $f^* \phi = \widetilde{f}^* \circ \WH{\phi} \circ \WH{\widetilde{f}}$
$$\xymatrix{A \ar[d]_{f^*\phi} \ar[r]_{\widetilde{f}} & B  \ar[d]_{\phi} \\ A^*  & \ar[l]^{{\widetilde{f}}^*} B^* }$$

\begin{prop}
    $f^* \phi$ defines a strong homotopy inner product on $A$ which is equivalent to $\phi$.
\end{prop}
\begin{proof}
  Since $\phi:B \to B^*$ is skew-symmetric and closed, so is $f^* \phi$ by lemma 5.6 of \cite{Cho}.
  It is not hard to check that $f^*\phi$ is also homologically non-degenerate as $f$ is a quasi-isomorphism.
  Hence, $f^*\phi$, by the proposition \ref{prop:shi} (1), is a strong homotopy inner product. Hence there exist
  an $\AI$-algebra $C$ which is cyclic symmetric ($\psi:C \to C^*$), and $\AI$-quasi-homomorphism $h:C \to A$ with the following commutative diagrams.

\begin{equation}
\xymatrix@C+1cm{ C \ar[r]_{ \widetilde{h}}  \ar[d]_\psi & \, A \ar[d]_{f^*\phi} \ar[r]_{\widetilde{f}}  & B \ar[d]_{\phi} \\ C^*  &  \ar[l]^{\widetilde{h}^*} A^* & \ar[l]^{\widetilde{f}^*} (B)^* }.
\end{equation}
  From the diagram, it is easy to see that $\phi$ and $f^*\phi$ is equivalent in the sense of definition
\ref{def:equiv}.
\end{proof}

\section{Potentials}
In this section we define a potential of a homotopy cyclic $\AI$-algebra and prove its properties.
Let $(A,m^{A}_{*})$ be given a strong homotopy inner product $\phi: A \to A^*$. Recall that an $\AI$-bimodule map $\phi$ is
given by a family of maps
$$\phi_{p,q}: A^{\otimes p} \otimes \underline{A} \otimes A^{\otimes q} \to A^*,$$
where the underlined $A$ is to emphasize that it is an $A$-bimodule for reader's convenience. We denote by
\begin{equation}\label{eq:pq}
<x_1,\cdots,x_p,\underline{v},y_1,\cdots,y_q|w>_{p,q} := \phi(x_1,\cdots,x_p,\underline{v},y_1,\cdots,y_q)(w).
\end{equation}
As in the cyclic case, let $e_i$ be generators of $A$ as a vector space, which is assumed to be finite dimensional.
(One may use pull-back defined in the previous section using the inclusion $\iota:H^*(A) \to A$ in the
case that $H^*(A)$ is finite dimensional).
Define $\bx = \sum_i e_i x_i$ where $x_i$ are formal parameters with $deg(x_i) = - deg(e_i)$.

Now we give a definition a potential for the strong homotopy inner products.
\begin{definition}\label{def:homotopypotential}
    The \textbf{potential} of an $\AI$-algebra $(A,m^{A}_{*})$ with a strong homotopy inner product $\phi: A \to A^*$ is defined as
    \begin{align}\label{stpot}
            \Phi ^{A}(\bx)
            & = \sum_{N=1}^{\infty} \Phi ^{A}_N(\bx)  \nonumber \\
            & := \sum_{N=1}^{\infty} \sum_{p+q+k=N}^{\infty} \frac{1}{N+1}<{\bx,\bx,\cdots , \bx}, \underline{m_{k}^{A}(\bx,\bx,\cdots,\bx)},\bx,\cdots,\bx|\bx >_{p,q}
    \end{align}
\end{definition}
The definition itself is somewhat similar to that of cyclic case(\ref{podef}). But in (\ref{podef}), the
fraction $1/k$ was to cancel out repetitive contribution to the potential due to cyclic symmetry (\ref{cycsym}), whereas in the strong homotopy case, such cyclic symmetry
of the rotation of arguments do not exist. Namely, in general
$$<e_1,\cdots,\UL{m_i(e_j,\cdots, e_{j+i-1})},\cdots,e_k|e_{k+1}> \neq <e_2,\cdots,\UL{m_i(e_{j+1},\cdots, e_{j+i})},\cdots,e_{k+1}|e_1>.$$
We later show that the combination of $\AI$-bimodule equation,
skew-symmetry and closed condition will compensate the absence of the strict cyclic symmetry.

We explain how the potential behaves under pull-backs, and this will show the relation between the potentials
of equivalent strong homotopy inner products. For $\AI$-quasi-isomorphism $h:B \to A$
the pull-back of a potential is defined as follows:
We assume $B$ is finite dimensional as a vector space, and denote by $\{f_*\}$ its basis, and
introduce corresponding formal variables $y_*$ as before.
Suppose
$$h_{k}(f_{j_1},\cdots,f_{j_k}) = h^i_{j_1,\cdots,j_k} e_i, \;\;\; h^i_{j_1,\cdots,j_k} \in \kk.$$
Then, we set
\begin{equation}\label{changeco}
	x_i \mapsto h^i_{j_{11}} y_{j_{11}} + h^i_{j_{21},j_{22}} y_{j_{21}} y_{j_{22}} + \cdots +
	h^{i}_{j_{l1},\cdots,j_{lk}} y_{j_{l1}}\cdots y_{j_{lk}} + \cdots.
\end{equation}
Then, one define the pull-back $h^*\Phi^A$ by using the above change of coordinate formula.
\begin{theorem}\label{thm:pot}
 Let $\phi:A \to A^*$ be a strong homotopy inner products. Let   $B$  be a cyclic $\AI$-algebra with a quasi-isomorphism
$h:B \to A$ providing the commutative diagram (\ref{defdiagram1}). Then, we have
    \begin{equation*}\Phi ^{B}=h^{*}\Phi ^{A} \end{equation*}
\end{theorem}
\begin{proof}
The overall scheme of the proof, which is first to differentiate and then to compare, follows that of \cite{Cho} (idea due to Kajiura \cite{Kaj} in the unfiltered case).
The main difficulty, and the essential part of the proof is the
first step where we take (formal) partial derivatives on each side.
The following lemma shows that after partial
differentiation, the fraction on each summand disappears.
\begin{lemma}\label{lem:fundlem}
    \begin{align*}
        \frac{\partial}{\partial x_i}\Phi ^{A}_N(\bx)
        & = \frac{\partial}{\partial x_i} \sum_{p+q+k=N}^{\infty} \frac{1}{N+1}<{\bx,\bx,\cdots , \bx}, \underline{m_{k}^{A}(\bx,\bx,\cdots,\bx)},\bx,\cdots,\bx|\bx >_{p,q} \\
        & = \sum_{k=1}^{\infty}<{\bx,\bx,\cdots , \bx}, \underline{m_{k}^{A}({\bx,\bx,\cdots,\bx})},{\bx,\cdots,\bx}|e_i >_{p,q}.
    \end{align*}
\end{lemma}
We assume the lemma for a moment and show the proof of the theorem using the lemma.
Let $\{f_i\}$ be basis of $H^* (A)$, and let $\{y_i\}$ be corresponding formal variables for $\{f_i\}$, namely $\by:=\sum_{i}y_i f_i.$

  We let $h(\boldsymbol{y}) := \sum_{k\geq 1}h_k ({\boldsymbol{y}}^{\otimes k})$. Then
  $$\frac{\partial}{\partial y_i}\Phi^{H^* (A)} = \sum_{k\geq 1}<m_k ^{H^* (A)}(\by, \cdots, \by), f_i>$$
  by cyclic symmetry, and
  \begin{eqnarray}
      \frac{\partial}{\partial y_i}h^* \Phi ^A &=& \frac{\partial}{\partial y_i}\sum_{\stackrel{k\geq 1}{p+q+1=N}}\frac{1}{N+1}<\tdy ^{\otimes p}, \underline{m_k ^A(\tdy, \cdots, \tdy)}, \tdy ^{\otimes q} | \tdy> \nonumber \\
      &=& \sum_{\stackrel{N\geq 1}{p+q+1=N}}<\tdy ^{\otimes p}, \underline{m_k ^A(\tdy, \cdots, \tdy)}, \tdy ^{\otimes q} | \frac{\partial}{\partial y_i}\tdy> \nonumber
  \end{eqnarray}
  by above lemma.
  From the diagram \ref{defdiagram1}, we  have $\psi = \widetilde{h}^* \circ \widehat{\phi} \circ \WH{\widetilde{h}}$, where all maps are $H^* (A)$-bimodule homomorphisms, consider following:
  \begin{eqnarray}
    & & \sum_{\stackrel{p,q\geq 0}{k\geq 1}} \psi(\by ^{\otimes p}, \underline{m_k^{H^* (A)}(\overrightarrow{\by})}, \by ^{\otimes q})(f_i) \label{eq:comdgm} \\
    &=& \sum_{\stackrel{p,q\geq 0}{k\geq 1}} (\widetilde{h}^* \circ \widehat{\phi} \circ \widetilde{h})(\by ^{\otimes p}, \underline{m_k^{H^* (A)}(\overrightarrow{\by})}, \by ^{\otimes q})(f_i) \nonumber \\
    &=& \sum_{\stackrel{p,q\geq 0}{k\geq 1}} \sum_{\stackrel{p_1+p_2+p_3=p}{q_1+q_2+q_3=q}} \widetilde{h}^* (\by ^{\otimes p_3}, \phi(\widehat{h}(\by ^{\otimes p_2}), \underline{h(\by ^{\otimes p_1}, \underline{m_k^{H^* (A)}(\overrightarrow{\by})}, \by ^{\otimes q_1})},\widehat{h}(\by ^{\otimes q_2})), \by ^{\otimes q_3})(f_i) \nonumber \\
    &=& \sum_{\stackrel{p,q\geq 0}{k\geq 1}} \sum_{\stackrel{p_1+p_2+p_3=p}{q_1+q_2+q_3=q}} \phi(\widehat{h}(\by ^{\otimes p_2}), \underline{h(\by ^{\otimes p_1}, \underline{m_k^{H^* (A)}(\overrightarrow{\by})}, \by ^{\otimes q_1})}, \widehat{h}(\by ^{\otimes q_2}))(h(\by ^{\otimes q_3}, \underline{f_i}, \by ^{\otimes p_3})) \nonumber \\
    &=& \sum_{\stackrel{p,q\geq 0}{k\geq 1}} \sum_{\stackrel{p_1+p_2+p_3=p}{q_1+q_2+q_3=q}} <\widehat{h}(\by ^{\otimes p_2}), \underline{h(\by ^{\otimes p_1}, \underline{m_k^{H^* (A)}(\overrightarrow{\by})}, \by ^{\otimes q_1})}, \widehat{h}(\by ^{\otimes q_2})|h(\by ^{\otimes q_3}, \underline{f_i}, \by ^{\otimes p_3})> \nonumber \\
    &=& \sum_{\stackrel{N\geq 1}{p+q+1=N}}<\tdy ^{\otimes p}, \underline{m_k ^A(\tdy, \cdots, \tdy)}, \tdy ^{\otimes q} | \frac{\partial}{\partial y_i}\tdy> \nonumber \\
    &=& \frac{\partial}{\partial y_i}h^* \Phi ^A \nonumber
  \end{eqnarray}

  Here, we denote by $m_k(\overrightarrow{\by})$ the expression $m_k(\by, \cdots, \by)$ for simplicity.
  The last identity holds because the sum is over all $p_1+p_2+p_3=p$ and $q_1+q_2+q_3=q$ where $p$ and $q$ run over all nonnegative integers, and there is $\AI$-bimodule relation $\widehat{m^A}\circ\widehat{h}=\widehat{h}\circ\widehat{m^{H^* (A)}}$.

  The summands of (\ref{eq:comdgm}) are all zero except for $(p,q)=(0,0)$ because $\psi$ is a cyclic symmetric inner product. Hence,
  $$\frac{\partial}{\partial y_i}h^* \Phi ^A = \sum_{k\geq 1} \psi(m_k^{H^* (A)}(\overrightarrow{\by}))(f_i) = \sum_{k\geq 1}<m_k ^{H^* (A)}(\by, \cdots, \by), f_i> = \frac{\partial}{\partial y_i}\Phi^{H^* (A)}.$$
  This proves the theorem.
\end{proof}
\begin{proof} We prove the lemma \ref{lem:fundlem}. Before we proceed, we give some remarks on the signs. The sign convention used in this paper and in \cite{Cho} is
the Koszul convention after the degree one shift. For simplicity, we omit the Koszul sign factor and
the expressions will appear with $+$ if it agrees with
the Koszul sign rule, $-$ if it is the negative of the Koszul sign. We illustrate this for two examples, from
which the general convention can be easily understood.
The first example is the $\AI$-equation with two inputs. We write
\begin{equation}\label{sgn:eq1}
m_1 m_2(x_1,x_2) + m_2(m_1(x_1),x_2) + m_2(x_1,m_1(x_2)) =0
\end{equation}
whereas the actual equation is
$$m_1 m_2(x_1,x_2) + m_2(m_1(x_1),x_2) +  (-1)^{|x_1|'} m_2(x_1,m_1(x_2)) =0.$$
The equation (\ref{sgn:eq1}) will be also written as
$$m_1 m_2(x_1,x_2) = - m_2(m_1(x_1),x_2) - m_2(x_1,m_1(x_2)).$$

The second example is the equation for $<m_2(\UL{x_1},x_2)|x_3>$.
Note that $\phi$ being $\AI$-bimodule map $\phi:A \to A^*$ with the induced $\AI$-bimodule
structure on $A^*$ (see expression (3.3) \cite{Cho} for the precise definition) implies the
following actual equation.
\begin{eqnarray*}
<m_2(\UL{x_1},x_2)|x_3> + <m_1(\UL{x_1}),x_2|x_3> + (-1)^{|x_1|'}<\UL{x_1},m_1(x_2)|x_3> \\
+ (-1)^{|x_1|'+|x_2|'}<\UL{x_1},x_2|m_1(x_3)> + (-1)^{|x_1|'}<\UL{x_1}|m_2(x_2,x_3)>=0.
\end{eqnarray*}
In this paper, the above equation will be written simply as
\begin{eqnarray*}
<m_2(\UL{x_1},x_2)|x_3> + <m_1(\UL{x_1}),x_2|x_3> + <\UL{x_1},m_1(x_2)|x_3> \\
+ <\UL{x_1},x_2|m_1(x_3)> + <\UL{x_1}|m_2(x_2,x_3)>=0.
\end{eqnarray*}

Now, we begin the proof of the lemma.
From now on, we replace $m_k ^A$ by $m_k$ if there is no ambiguity. By taking a derivative, the expression becomes:
    \begin{eqnarray}
        & & \frac{\partial}{\partial x_i} \sum_{p+q+k=N}^{\infty}<{\bx,\bx,\cdots , \bx}, \underline{m_{k}^{A}(\bx,\bx,\cdots,\bx)},\bx,\cdots,\bx|\bx >_{p,q} \label{eq:2.0} \\
        &=& \sum_{\stackrel{p+q+k=N}{r+s=k-1}}<\bx , \cdots , \bx , \underline{m_k(\overbrace{\bx, \cdots, \bx}^{r}, e_i, \overbrace{\bx, \cdots, \bx}^{s})}, \bx , \cdots , \bx|\bx>_{p,q} \label{eq:2.1} \\
        &+& \sum_{\stackrel{p+q+k=N}{r+s=p-1}}<\underbrace{\bx , \cdots , \bx}_{r} , e_i , \underbrace{\bx , \cdots , \bx}_{s} , \underline{m_k(\bx, \cdots, \bx)} , \bx , \cdots , \bx| \bx>_{p,q} \label{eq:2.2} \\
        &+& \sum_{\stackrel{p+q+k=N}{r+s=q-1}}<\bx , \cdots , \bx , \underline{m_k(\bx, \cdots, \bx)}, \underbrace{\bx , \cdots , \bx}_{r} , e_i, \underbrace{\bx , \cdots , \bx}_{s}| \bx>_{p,q} \label{eq:2.3} \\
        &+& \sum_{p+q+k=N}<\bx , \cdots ,\bx , \underline{m_k(\bx, \cdots, \bx)}, \bx , \cdots , \bx| e_i>_{p,q} \label{eq:2.4}.
    \end{eqnarray}
Now, the lemma can be proved by the following lemma.
\end{proof}
\begin{lemma}\label{claim}
The sum of the terms in (\ref{eq:2.1}),(\ref{eq:2.2}) and (\ref{eq:2.3}) equals to $N$ times of the expression (\ref{eq:2.4})
\end{lemma}
\begin{proof}
To prove the lemma, we recall the notion of $\AI$-bimodule equation.
  The equation for $\AI$-bimodule homomorphism $A \rightarrow A^*$ is
\begin{equation}\label{AIBeq1}
 \phi \circ \WH{b_A} =b_{A^*} \circ \WH{\phi}
\end{equation}
 with $b_A=m^A$ when $A$ is considered to be an $\AI$-bimodule, and $b_{A^*}$ is defined by canonical construction of the dual of the $\AI$-bimodule $A$. Here $\WH{\phi}$ is
  the coalgebra homomorphism induced from $\phi$ (We refer readers to \cite{Cho},\cite{T} or \cite{GJ} for details).
  Let us restrict the equation (\ref{AIBeq1}) to the case $(\bx , \cdots , \bx , \underline{e_i} , \bx , \cdots , \bx) \in A^{\otimes n}\otimes \underline{A} \otimes A^{\otimes m}$ where $n+m+1=N$. Then it becomes
    \begin{eqnarray}
        & & \sum_{\stackrel{p+j_1=n}{j_2+q=m}}<\bx , \cdots , \bx , \underline{m_{j_1+j_2+1}(\overbrace{\bx, \cdots, \bx}^{j_1}, \underline{e_i}, \overbrace{\bx, \cdots, \bx}^{j_2})}, \bx , \cdots , \bx|\bx>_{p,q}  \label{eq:5} \\
        &+& \sum_{\stackrel{k_1+k_2+j=n}{p=k_1+k_2+1}}<\underbrace{\bx , \cdots , \bx}_{k_1} , m_{j}(\bx, \cdots, \bx), \underbrace{\bx , \cdots , \bx}_{k_2} , \underline{e_i}, \bx , \cdots , \bx|\bx>^{dum}_{p,m}  \label{eq:6} \\
        &+& \sum_{\stackrel{l_1+l_2+h=m}{q=l_1+l_2+1}}<\bx , \cdots , \bx , \underline{e_i} , \underbrace{\bx , \cdots , \bx}_{l_1} , m_h(\bx, \cdots, \bx), \underbrace{\bx , \cdots , \bx}_{l_2}| \bx>^{dum}_{n,q} \label{eq:7}\\
        &=& \sum_{\stackrel{p+k_1=m}{k_2+q=n}}<\bx , \cdots , \bx , \underline{m_{k_1+k_2+1}(\overbrace{\bx, \cdots, \bx}^{k_1}, \underline{\bx}, \overbrace{\bx \cdots \bx}^{k_2})}, \bx , \cdots , \bx| e_i>_{p,q}. \label{eq:8}
    \end{eqnarray}
    \begin{remark}\label{rem:1}
    It is important to note that the expression in the summand (\ref{eq:8}) is obtained in $k:=k_1+k_2+1$ different ways according to the position of the (underlined) bimodule element $\underline{x}$. Namely, different choices of a bimodule element still give rise to equivalent expressions.
    \end{remark}
    \begin{remark}\label{rem:12}
    Here the terms (\ref{eq:5}) and (\ref{eq:8}) in the above $\AI$-bimodule equation do appear in the process of derivation (\ref{eq:2.0}) but
    the terms (\ref{eq:6}) and (\ref{eq:7}) do not appear in (\ref{eq:2.0}). Hence we marked them as $<,>^{dum}$ for reader's convenience to
    indicate that they are dummy parts. We will show how all the dummy parts are canceled out or used in the subsequent process.
    \end{remark}
    We say an expression such as in (\ref{eq:5}), $ \cdots$, (\ref{eq:8}) to be of $(n,m)$-type as it is obtained
    from the input $A^{\otimes n}\otimes \underline{A} \otimes A^{\otimes m}$. And for convenience, we will denote
    the summands as in (\ref{eq:5}), $ \cdots$, (\ref{eq:8}) to be $\sum_{(n,m)-type}$ instead of writing down specific conditions.

Note that the expression  (\ref{eq:5})  equals (\ref{eq:2.1}) and (\ref{eq:8}) provides $k$ times the expression (\ref{eq:2.4}) from the
remark \ref{rem:12}. Hence, we may use the above $\AI$-bimodule equation to turn (\ref{eq:2.1}) into
$k$-times (\ref{eq:2.4}) together with dummy terms. Hence, to prove the Lemma \ref{claim}, we need to find $N-k$ times the expression (\ref{eq:2.4})
from what are left out in (\ref{eq:2.0}) together with the new dummy terms.

 Now, we explain the dummy expressions we add to the equation.
 We set $m_j(\vec{\bx}):=m_j(\bx, \cdots, \bx)$ just for simplicity.
 The following are the dummy expressions to be added to the expression (\ref{eq:2.0}).

 \begin{eqnarray*}
       (i)& \sum_{p+j+k_1+k_2+1=N} & <\bx^{\otimes p}, m_j(\vec{\bx}), \bx^{\otimes k_1} , \ue , \bx^{\otimes k_2} | \bx>^{dum},\\
        (ii) & \sum_{p+j+k_1+k_2+k_3+2=N } & <\bx^{\otimes p} , e_i , \bx^{\otimes k_1} , m_j(\vec{\bx}) , \bx^{\otimes k_2}, \ux , \bx^{\otimes k_3} | \bx>^{dum},\\
        (iii)& \sum_{p+j+k_1+k_2+k_3+2=N} & <\bx^{\otimes p}, m_j(\vec{\bx}) , \bx^{\otimes k_1}  , e_i , \bx^{\otimes k_2} , \ux , \bx^{\otimes k_3} | \bx>^{dum}, \\
        (iv) & \sum_{p+j+k_1+k_2+k_3+2=N }& <\bx^{\otimes p} , e_i ,  \bx^{\otimes k_1} , \ux , \bx^{\otimes k_2} , m_j(\vec{\bx}) , \bx^{\otimes k_3} | \bx>^{dum},\\
        (v)& \sum_{ p+k_1+k_2+l_1+l_2+2=N}& <\bx^{\otimes p} , m_{j}(\bx^{\otimes l_1}, e_i,\bx^{\otimes l_2}), \bx^{\otimes k_1} , \ux , \bx^{\otimes k_2} | \bx>^{dum}, \\
        (vi)& \sum_{  p+j+k_1+k_2+1=N} & <\bx^{\otimes p} , m_j(\vec{\bx}) , \bx^{\otimes k_1} , \ux , \bx^{\otimes k_2}| e_i>^{dum},\\
        (vii)& \sum_{p+j+k_1+k_2+1=N } & <\bx^{\otimes p} , \ue , \bx^{\otimes k_1} ,m_j(\vec{\bx}), \bx^{\otimes k_2} | \bx>^{dum},\\
        (viii)& \sum_{p+j+k_1+k_2+k_3+2=N }& <\bx^{\otimes p} , \ux , \bx^{\otimes k_1} , e_i , \bx^{\otimes k_2},m_j(\vec{\bx}) , \bx^{\otimes k_3}| \bx>^{dum}, \\
        (ix)& \sum_{ p+j+k_1+k_2+k_3+2=N} & <\bx^{\otimes p} , \ux , \bx^{\otimes k_1} , m_j(\vec{\bx}), \bx^{\otimes k_2} , e_i , \bx^{\otimes k_3} | \bx>^{dum},\\
        (x)& \sum_{ p+j+k_1+k_2+k_3+2=N} & <\bx^{\otimes p} , m_j(\vec{\bx}) , \bx^{\otimes k_1} , \ux , \bx^{\otimes k_2} , e_i , \bx^{\otimes k_3} | \bx>^{dum},\\
        (xi)& \sum_{p+k_1+k_2+l_1+l_2+2=N } & <\bx^{\otimes p} , \ux , \bx^{\otimes k_1} , m_{j}(\bx^{\otimes l_1},e_i,\bx^{\otimes l_2}) , \bx^{\otimes k_2}|\bx>^{dum}, \\
        (xii)& \sum_{p+j+k_1+k_2+1=N  } & <\bx^{\otimes p} , \ux , \bx^{\otimes k_1} , m_j(\vec{\bx}) , \bx^{\otimes k_2}| e_i>^{dum},
    \end{eqnarray*}

The following is easy to check.
\begin{lemma}\label{lem:eq}
By applying skew symmetry condition, we have
\begin{eqnarray*}
(i)+(xii)=0,& (ii)+(ix)=0, &(iii)+(viii)=0,\\
(iv)+(x)=0, & (v)+(xi)=0, & (vi)+(vii)=0
\end{eqnarray*}
Hence, the overall sum also vanishes:
$$(i)+ (ii)+ \cdots +(xi)+(xii)=0. $$
\end{lemma}
Hence we add all the dummy terms above to the expression (\ref{eq:2.0}) without changing the value.
Notice that the expression $(i)$ and $(vii)$ already appeared in (\ref{eq:6}) and (\ref{eq:7}) and is used
to turn (\ref{eq:2.1}) into $k$ times (\ref{eq:2.4}).

In addition we need the following $\AI$-bimodule equation (\ref{AIBeq1}) which is obtained considering the case that
the input $e_i$ is {\it not} the bimodule element of the expression: Given $m,n \in \NN$ with $n+m = N-1$, we have
    \begin{eqnarray}
        & & \sum_{(n,m)-type}\big(<\bx , \cdots , \bx , e_i , \bx , \cdots , \bx , m_j(\vec{\bx}) , \bx , \cdots , \bx , \ux , \bx , \cdots, \bx | \bx>^{dum} \label{eq:111} \\ 
        &+& <\bx , \cdots , \bx , m_j(\vec{\bx}) , \bx , \cdots , \bx , e_i , \bx , \cdots , \bx , \ux , \bx , \cdots, \bx| \bx>^{dum} \label{eq:112} \\ 
        &+& <\bx , \cdots , \bx , e_i , \bx , \cdots , \bx , \ux , \bx , \cdots , \bx , m_j(\vec{\bx}) , \bx , \cdots, \bx| \bx>^{dum} \label{eq:113} \\ 
        &+& <\bx , \cdots , \bx , m_j(\bx, \cdots, \bx, e_i, \bx, \cdots, \bx), \bx , \cdots , \bx , \ux , \bx , \cdots, \bx |\bx>^{dum} \label{eq:114} \\ 
        &+& <\bx , \cdots , \bx , \ux , \bx , \cdots , \bx , e_i , \bx , \cdots , \bx , m_j(\vec{\bx}), \bx , \cdots, \bx| \bx>^{dum} \label{eq:115} \\ 
        &+& <\bx , \cdots , \bx , \ux , \bx , \cdots , \bx , m_j(\vec{\bx}), \bx , \cdots , \bx , e_i , \bx , \cdots, \bx| \bx>^{dum} \label{eq:116} \\ 
        &+& <\bx , \cdots , \bx , m_j(\vec{\bx}) , \bx , \cdots , \bx , \ux , \bx , \cdots , \bx , e_i , \bx , \cdots, \bx| \bx>^{dum} \label{eq:117} \\  
        &+& <\bx , \cdots , \bx , \ux , \bx , \cdots , \bx , m_j(\bx, \cdots, \bx, e_i, \bx, \cdots, \bx) , \bx , \cdots, \bx| \bx>^{dum} \label{eq:118} 
        \big)
         \end{eqnarray} 
    \begin{eqnarray}
        &=& -\sum_{(n,m)-type} \big(<\bx , \cdots , \bx , e_i , \bx , \cdots , \ux ,  \cdots, \bx|m_j(\bx,\cdots,\ux,\cdots,\bx)>
         \label{eq:irr3} \\
       &+& <\bx , \cdots , \bx , e_i , \bx , \cdots , \bx , \underline{m_j(\bx, \cdots, \bx, \ux, \bx,\cdots,\bx)} , \bx , \cdots, \bx|\bx> \label{eq:rel1} \\
        &+& <\bx , \cdots , \bx , \underline{m_j(\bx,\cdots,\bx, \ux, \bx,\cdots,\bx)} , \bx , \cdots , \bx , e_i , \bx , \cdots, \bx |\bx> \big) \label{eq:rel2}       \\
        &+& <\bx , \cdots , \ux , \bx , \cdots , \bx , e_i , \bx , \cdots, \bx|m_j(\bx,\cdots,\bx, \ux, \bx,\cdots,\bx)> \label{eq:irr4} \\
        &+& <\bx , \cdots , \bx , \ux , \bx , \cdots , \bx|m_j(\bx,\cdots,\bx, \ux, \bx,\cdots,\bx, e_i, \bx,\cdots,\bx)> \label{eq:irr5} \\
        &+& <\bx , \cdots , \bx , \ux , \bx , \cdots , \bx |m_j(\bx,\cdots,\bx, e_i, \bx,\cdots,\bx, \ux, \bx,\cdots,\bx)> \label{eq:irr6} \\
        &+& <\bx , \cdots , \bx , \underline{m_j(\bx,\cdots,\bx, e_i,\bx, \cdots, \bx,\ux, \bx,\cdots,\bx)} , \bx , \cdots, \bx|\bx> \label{eq:irr7} \\
        &+& <\bx , \cdots , \bx , \underline{m_j(\bx,\cdots, \bx,\ux, \bx,\cdots, \bx,e_i, \bx,\cdots,\bx)} , \bx , \cdots, \bx|\bx> \big) \label{eq:irr8}
    \end{eqnarray}
In the case that $n=m$, we consider only the above equation, but if $n \neq m$, then
we also consider the similar equation for $(m,n)$-type also.
Observe the following facts.
\begin{itemize}
      \item For later purposes, we have separated certain dummy expressions in the LHS.
      \item Some of dummy expressions already appeared before.
      $$(\ref{eq:111}) \sim (ii), (\ref{eq:112}) \sim (iii),(\ref{eq:113})\sim (iv),(\ref{eq:114}) \sim (v),$$
      $$(\ref{eq:115}) \sim (viii),(\ref{eq:116}) \sim (ix),   (\ref{eq:117}) \sim (x),(\ref{eq:118})\sim (xi).$$
            \item As $(i),(vii)$ has been used in (\ref{eq:6}),(\ref{eq:7}), the dummy expressions (among the list above the Lemma \ref{lem:eq} which has not used elsewhere (yet)  are $(vi)$ and $(xii)$.
      \item Note the equivalence in (\ref{eq:irr7}) and (\ref{eq:irr8}) if the number of $\bx$ ahead of $e_i$ in the following expression
      are the same:
      $$\underline{m_j(\bx,\cdots,\bx, e_i,\bx, \cdots, \bx,\ux, \bx,\cdots,\bx)} =  \underline{m_j(\bx,\cdots, \bx,\ux, \bx,\cdots, \bx,e_i, \bx,\cdots,\bx)}$$

        \end{itemize}
%

Now, we show that in the right hand side of the equation, all the terms cancel out from skew symmetry.
In the case that $n=m$, we have the cancellations (from skew symmetry)
$$ (\ref{eq:irr3}) + (\ref{eq:rel2}) =0, \;  (\ref{eq:irr4}) + (\ref{eq:rel1}) =0, (\ref{eq:irr5})+(\ref{eq:irr7})=0,(\ref{eq:irr6})+(\ref{eq:irr8})=0. $$

Hence we may consider the case that $n\neq m$. In this case, there are similar cancellations between $(n,m)$-type and $(m,n)$-type terms.
Namely, if we use subscript $\,_{(n,m)},\,_{(m,n)}$ to denote $(n,m)$-type and $(m,n)$-type, we have
$$ (\ref{eq:irr3})_{(n,m)} + (\ref{eq:rel2})_{(m,n)} =0,\;\;
(\ref{eq:irr3})_{(m,n)} + (\ref{eq:rel2})_{(n,m)} =0,$$
$$ (\ref{eq:irr4})_{(n,m)} + (\ref{eq:rel1})_{(m,n)} =0,\;\;
(\ref{eq:irr4})_{(m,n)} + (\ref{eq:rel1})_{(n,m)} =0,$$
$$(\ref{eq:irr5})_{(m,n)}+(\ref{eq:irr7})_{(n,m)}=0,\;\;
(\ref{eq:irr5})_{(n,m)}+(\ref{eq:irr7})_{(m,n)}=0,$$
$$(\ref{eq:irr6})_{(m,n)}+(\ref{eq:irr8})_{(n,m)}=0, \;\;
(\ref{eq:irr6})_{(n,m)}+(\ref{eq:irr8})_{(m,n)}=0.$$
Hence, the right hand side always vanishes.
)
Consequently, if we collect all the remaining terms, there are
(\ref{eq:2.2}) and (\ref{eq:2.3}) and $(vi)$ and $(xii)$.
Now, we show that the addition of these expressions produce $N-k$ times (\ref{eq:2.3}), which proves the theorem.
Let us list the remaining terms first.
    \begin{eqnarray}
    (\ref{eq:2.3})  &  \sum_{p+k+j_1+j_2+1=N} &  <\bx^{\otimes p}, \underline{m_k(\vec{\bx})} , \bx^{\otimes j_1}, e_i ,\bx^{\otimes j_2} | \bx>, \\
     (\ref{eq:2.2})  &  \sum_{p+k+j_1+j_2+1=N}  &  <\bx^{\otimes p},e_i , \bx^{\otimes j_1}, \underline{m_k(\vec{\bx})} , \bx^{\otimes j_2} | \bx>, \\
      (vi)  &  \sum_{p+k+j_1+j_2+1=N}   & <\bx^{\otimes p}, m_k(\vec{\bx}) ,\bx^{\otimes j_1}, \ux ,\bx^{\otimes j_2} | e_i>, \\
        (xii)  &  \sum_{p+k+j_1+j_2+1=N} & <\bx^{\otimes p}, \ux , \bx^{\otimes j_1}, m_k(\vec{\bx}) , \bx^{\otimes j_2} | e_i>.
    \end{eqnarray}

Now we use closed condition with these terms to obtain (\ref{eq:2.4}).
\begin{enumerate}
        \item  By applying the closed condition in the theorem \ref{thm:shi} to (\ref{eq:2.3}) and $(xii)$, we obtain
 (here $(a_i,a_j,a_k)$ corresponds to $(e_i,m_k(\vec{x}),\bx)$)
            \begin{eqnarray*}
                & & <\underbrace{\bx , \cdots , \bx , \ux , \bx , \cdots , \bx}_{s} , m_k(\vec{\bx}) , \bx ^{\otimes r}|e_i> \\
                &+& <\bx , \cdots , \bx , \underline{m_k(\vec{\bx})} , \bx , \cdots , \bx , e_i, \bx , \cdots , \bx|\bx> \\
                &+& <\bx ^{\otimes r} , \underline{e_i} , \bx ^{\otimes s}|m_k(\vec{\bx})>  =0
            \end{eqnarray*}
        In fact, we obtain $s$ different such equations depending on the position of $\UL{\bx}$ in the first line.
        Hence, the sum of expressions  (\ref{eq:2.3}) and $(xii)$ produces $s$ times that of (\ref{eq:2.4})
        as the last term equals the minus of (\ref{eq:2.4}):
        $$<\bx ^{\otimes r} , \underline{e_i} , \bx ^{\otimes s}|m_k(\vec{\bx})>
        = - <\bx ^{\otimes s}, \underline{m_k(\vec{\bx})},\bx ^{\otimes r}|\underline{e_i}$$
        \item Similarly by applying the closed condition to (\ref{eq:2.2}) and $(vi)$,
            \begin{eqnarray*}
                & & <\bx ^{\otimes s} , m_k(\vec{\bx}) , \underbrace{\bx , \cdots , \bx, \ux , \bx , \cdots , \bx}_{r}|e_i> \\
                &+& <\bx , \cdots , \bx , e_i , \bx , \cdots , \bx , \underline{m_k(\vec{\bx})} , \bx , \cdots , \bx| \bx> \\
                &+& <\bx ^{\otimes r} , \underline{e_i} , \bx ^{\otimes s}|m_k(\vec{\bx})> \\
                &=& 0.
            \end{eqnarray*}
            we obtain $r$ different such equations depending on the position of $\UL{\bx}$ in the first line.
    \end{enumerate}
Hence we obtain $r+s = N-k$ times the expression of (\ref{eq:2.4}), which proves the lemma \ref{claim}.
\end{proof}

\section{Potential $\Psi$ and the generalized holonomy map}
In this section, we consider another potential $\Psi$ defined in the definition \ref{podef2} for
a unital homotopy cyclic $\AI$-algebra.
We discuss its gauge invariance and its relationship with the algebraic analogue of generalized holonomy map in \cite{ATZ}.

Let us first recall the definition of a unit for $\AI$-algebra.
\begin{definition}\label{def:unit}
An element $I \in C^0 = C^{-1}[1]$ is called a unit if
\begin{equation}\label{unit}
\begin{cases}
 m_{k+1}(x_1,\cdots,I,\cdots,x_k) = 0 \;\; \textrm{for} \; k\geq 2 \;\;\textrm{or}\;\; k =0 \;\; \\
m_2(I,x) = (-1)^{deg \, x} m_2(x,I) = x.
\end{cases}
\end{equation}
\end{definition}
We assume that the strong homotopy inner product $\phi:A \to A^*$ is an unital $\AI$-bimodule map, or
$\phi_{k,l}(\vec{a},v,\vec{b})(w)$ vanishes if one of $a_i$'s or $b_i$'s is a constant multiple of $I$.

We also recall the Maurer-Cartan elements and its gauge equivalences.
\begin{definition}
Let $A$ be an $\AI$-algebra. An element $b \in A^1$ satisfying $m(e^b)= \sum_{k}m_k(a,...,a)=0$
is called the Maurer-Cartan elements and we denote by $MC(A)$ the set of all Maurer-Cartan elements.
Let $\mathcal{MC}:=MC/\sim$ be the moduli space of Maurer-Cartan elements, whose gauge equivalence is defined
as follows(definition 2.3 of \cite{Fu}): $b$ is \textit{gauge equivalent} to $\WT{b}$ if there are one-parameter families $b(t)\in A^1[t],c(t)\in A^0[t]$ such that
\begin{itemize}
  \item $b(0)=b,b(1)=\WT{b},$ and
  \item $\displaystyle \frac{d}{dt}b(t)=\sum_{k \geq 1}m_k(b(t),...,b(t),c(t),b(t),...,b(t)).$
\end{itemize}
\end{definition}
We remark that $b(t)$ is also a Maurer-Cartan element for any $t$ (Lemma 4.3.7 of \cite{FOOO}).
Now, we prove the gauge invariance of the potential $\Psi$ for Maurer-Cartan elements.
\begin{prop}\label{prop:inv}
The potential $\Psi(x)=\sum_{p,q\geq 0}\frac{1}{p+q+1}<x^{\otimes p}\otimes\UL{x}\otimes x^{\otimes q}|I>$ when restricted to the Maurer-Cartan elements $MC$ is invariant under gauge equivalences.
 i.e. if $x(t)$ is a one-parameter family in the Maurer-Cartan solution space, then $$\frac{d}{dt}\Psi(x(t))=0.$$
\end{prop}
\begin{proof}
We prove this proposition with the help of following lemmas.
\begin{lemma}\label{lem:nofrac}
$\Psi(x)$ equals the following expression.
  $\Psi(x)=\sum_{k\geq 0}<\UL{x}\otimes x^{\otimes k}|I>.$
\end{lemma}
\begin{proof}
  By the closedness condition of $\phi$, for any $p$ and $q$ we have
  \begin{eqnarray*}
    <x^{\otimes p}\otimes \UL{x}\otimes x^{\otimes q}|I>&+&<x^{\otimes p+q}\otimes \UL{I}|x>\\
    &+&<x^{\otimes q}\otimes I \otimes \UL{x}\otimes x^{\otimes p-1}|x>=0.
  \end{eqnarray*}
  By definition of unital $\AI$-bimodule homomorphisms, we have $$<x^{\otimes q}\otimes I \otimes \UL{x}\otimes x^{\otimes p-1}|x>=0,$$ and the above equation gives
  $$<x^{\otimes p}\otimes \UL{x}\otimes x^{\otimes q}|I>= - <x^{\otimes p+q}\otimes \UL{I}|x> = <\UL{x}\otimes x^{\otimes p+q}|I>,$$
  where the last equality follows from the skew-symmetry of $\phi$. This proves the lemma.
\end{proof}
\begin{lemma}\label{lem:cycsumzero}
  $\sum_{\sigma \in \mathbb{Z}/n\ZZ}<\UL{a_{\sigma(1)}},a_{\sigma(2)},...,a_{\sigma(n-1)}|a_{\sigma(n)}>=0.$
\end{lemma}
\begin{proof}
Fix $a_1,\cdots,a_n$ and denote $[i,j]:=<...,\UL{a_i},...|a_j>.$ Then what we need to prove is $$[1,n]+[2,1]+\cdots+[n,n-1]=0.$$ The closedness condition of strong homotopy inner products gives $$[i,j]+[j,k]=[i,k].$$ Hence, it follows that $$[1,n]+[n,n-1]+\cdots+[2,1]=[1,n]+[n,1]=0.$$
\end{proof}
Now we prove the above proposition. First, assume $$\frac{d}{dt}x(t)=\sum_{i+j=k\geq 0}m_{k+1}(x(t)^{\otimes i}\otimes c(t) \otimes x(t)^{\otimes j}).$$ We denote $x$ by $x(t)$ and $c$ by $c(t)$, for it causes no problem in this proof. 

Applying lemma \ref{lem:nofrac}, the fraction disappears and we get
  \begin{eqnarray}
    \frac{d}{dt}\Psi(x)
    &=&\sum_{l\geq 0}<\UL{\sum_{i+j=k\geq 0}m_{k+1}(x^{\otimes i}\otimes c\otimes x^{\otimes j})}\otimes x^{\otimes l}|I> \label{eq:psider1} \\
   &+&\sum_{l,m\geq 0}<\UL{x}\otimes x^{\otimes l}\otimes \sum_{i+j=k\geq 0}m_{k+1}(x^{\otimes i}\otimes c\otimes x^{\otimes j})\otimes x^{\otimes m}|I>. \label{eq:psider2}
  \end{eqnarray}
  To prove that it is zero, we use the $\AI$-bimodule equation. Namely, we compute
  $$(\phi\circ \WH{m}-m^*\circ \WH{\phi})(\sum_{l\geq 0}\UL{c}\otimes x^{\otimes l}+\sum_{l,m\geq 0}\UL{x}\otimes x^{\otimes l}\otimes c\otimes x^{\otimes m})(I),$$
  which is a priori zero.
  \begin{eqnarray}
     (\phi\circ\WH{m})(\sum_{i\geq 0}\UL{c}\otimes x^{\otimes i})(I)
    &=&\sum_{l\geq 0}<\UL{\sum_{k\geq 0}m_{k+1}(c\otimes x^{\otimes k})}\otimes x^{\otimes l}|I> \label{eq:part1} \\
    &+&\sum_{l,m\geq 0}<\UL{c}\otimes x^{\otimes l}\otimes(\sum_{k\geq 1}m_k(x^{\otimes k}))\otimes x^{\otimes m}|I> \label{eq:MCzero1}
  \end{eqnarray}
  and (\ref{eq:MCzero1}) is zero by Maurer-Cartan equation.

  \begin{eqnarray}
      & &(\phi\circ \WH{m})(\sum_{i,j\geq 0}\UL{x}\otimes x^{\otimes i}\otimes c\otimes x^{\otimes j})(I) \nonumber \\
      &=&\sum_{l,m\geq 0}<\UL{\sum_{k\geq 1}m_k(x^{\otimes k})}\otimes x^{\otimes l}\otimes c\otimes x^{\otimes m}|I> \label{eq:MCzero2} \\
      &+&\sum_{l\geq 0}<\UL{\sum_{i\geq 1,j\geq 0}m_k(x^{\otimes i}\otimes c\otimes x^{\otimes j})}\otimes x^{\otimes l}|I> \label{eq:part2} \\
      &+&\sum_{l,m\geq 0}<\UL{x}\otimes x^{\otimes l}\otimes \sum_{i+j=k\geq 0}m_{k+1}(x^{\otimes i}\otimes c\otimes x^{\otimes j})\otimes x^{\otimes m}|I> \label{eq:part3} \\
      &+&\sum_{l,m,n\geq 0}<\UL{x}\otimes x^{\otimes l}\otimes c\otimes x^{\otimes m}\otimes \sum_{k\geq 1}m_k(x^{\otimes k})\otimes x^{\otimes n}|I>. \label{eq:MCzero3}
  \end{eqnarray}
  Remark again, that (\ref{eq:MCzero2}) and (\ref{eq:MCzero3}) vanish by Maurer-Cartan equation. Observe also that
  \begin{itemize}
    \item (\ref{eq:part1})+(\ref{eq:part2})=(\ref{eq:psider1}),
    \item (\ref{eq:part3})=(\ref{eq:psider2}).
  \end{itemize}
  It remains to show that $$(m^*\circ \WH{\phi})(\sum_{l\geq 0}\UL{c}\otimes x^{\otimes l}+\sum_{l,m\geq 0}\UL{x}\otimes x^{\otimes l}\otimes c\otimes x^{\otimes m})(I)=0.$$
  Since $I$ is the unit, we may easily verify that
  \begin{equation}
    (m^*\circ \phi)(\sum_{l\geq 0}\UL{c}\otimes x^{\otimes l})(I) =\sum_{l\geq 0}<\UL{c}\otimes x^{\otimes l}|x>, \label{eq:mphi1}
  \end{equation}
  \begin{equation}
    (m^*\circ \phi)(\sum_{l\geq 0}\UL{x}\otimes x^{\otimes l}\otimes c)(I) =\sum_{l\geq 0}<\UL{x}\otimes x^{\otimes l}|c>, \label{eq:mphi2}
  \end{equation}
  \begin{equation*}
    (m^*\circ \phi)(\sum_{l\geq 0,m\geq 1}\UL{x}\otimes x^{\otimes l}\otimes c\otimes x^{\otimes m})(I) =\sum_{l,m\geq 0}<\UL{x}\otimes x^{\otimes l}\otimes c\otimes x^{\otimes m}|x>.
  \end{equation*}
  In (\ref{eq:mphi1}) and (\ref{eq:mphi2}), for $l=0$, we have $$<c|x>+<x|c>=0$$ by skew-symmetry. For remaining parts, we collect terms appropriately and use closedness condition to show that they all vanish. More precisely, for $k\geq 1$, we claim that
 $$<\UL{c}\otimes x^{\otimes k}|x>+<\UL{x}\otimes x^{\otimes k}|c>
  +\sum_{l+m=k-1}<\UL{x}\otimes x^{\otimes l}\otimes c\otimes x^{\otimes m}|x>  =0$$

  But this follows from the previous lemma \ref{lem:cycsumzero}, by setting $a_1=c,a_2=\cdots=a_{k+2}=x.$
\end{proof}

Now, we discuss the potential $\Psi$ and the algebraic generalized holonomy map.
We refer readers to \cite{ATZ} or \cite{CL} for the relevant definitions of this construction.

First, recall from Proposition 6.1 of \cite{CL} that given a negative cyclic cohomology class $\alpha$ of an $\AI$-algebra $A$, one obtains a bimodule map $\WT{\alpha}:A \to A^*$. This provides a strong homotopy inner product, if $\alpha$ is
in addition homologically non-degenerate.
The definition \ref{podef2} thus provides the potential $\Psi^\alpha$ using $\alpha$.
Combined with the above proposition, we prove
\begin{theorem}
The potential $\Psi$ provides a map $\Psi:HC_{-}^\bullet (A)\to \mathcal{O}(\mathcal{MC})$
defined by $\alpha \mapsto \Psi^\alpha|_{MC}$. Furthermore, this agrees with the algebraic
analogue of generalized holonomy map of Abbaspour, Tradler and Zeinalian \cite{ATZ}.
\end{theorem}
\begin{proof}
We only need to prove the relation with that of \cite{ATZ} and we 
recall the construction of a map $\rho:HC_{-}^\bullet (A)\to \mathcal{O}(\mathcal{MC})$.
Here we always work with reduced versions of negative cyclic or Hochschild (co)homologies.

Given a Maurer-Cartan element $a$ of a unital $\AI$-algebra $A$, consider the expression (Definition 8 of \cite{ATZ})
$$P(a):=\sum_{i \geq 0} I \otimes a^{\otimes i} = (I \otimes I) + (I \otimes a) + (I \otimes a \otimes a) + \cdots. $$
One can check that $P(a)$ is a Hochschild homology cycle from the unital property of $I$ and the Maurer-Cartan equation.
Note that Connes-Tsygan operator $B$ of $P(a)$ vanishes on the reduced complex, due to the unit $I$. Hence, $P(a)$ can be
considered as a negative cyclic homology cycle. 

Hence, given a negative cyclic cohomology cycle $\alpha \in HC_{-}^\bullet (A)$,  one can use the pairing 
$<,>:HC^\bullet_{-}(A) \otimes HC_\bullet^-(A) \to \kk$ to define the map $\rho$ as
\begin{equation}\label{dglacase}\rho([\alpha])([a]):=\langle\alpha,\sum_{i\geq 0}I\otimes a^{\otimes i}\rangle\end{equation}

Now, we compare the above expression with that of Lemma \ref{lem:nofrac}. 
We recall the following proposition from \cite{CL}.
\begin{prop}[Proposition 6.1 \cite{CL}]
    Let $\alpha \in C_{red}^{\bullet}(A,A^*)$ be a negative cyclic cocycle.
    We define $$\widetilde{\alpha_0}(\vec{a}, \underline{v}, \vec{b})(w):=\alpha_0(\vec{a},v,\vec{b})(w)-\alpha_0(\vec{b},w,\vec{a})(v).$$
    Then $\widetilde{\alpha_0}$ is an $\AI$-bimodule map from $A$ to $A^*$,
    satisfying the skew-symmetry and closedness condition.
\end{prop}
Here $\alpha_0$ is the part of $\alpha$ which is dual to the inclusion of Hochschild homology cycles.
Also, from the unital property, we have
$$\widetilde{\alpha_0}(\UL{a},a,\cdots,a)(I)=\alpha_0(a,\cdots,a)(I) - \alpha_0(a,\cdots,a,I)(a)= \alpha_0(a,\cdots,a)(I)$$
Hence, 
$$\langle\alpha,I\otimes a^{\otimes i}\rangle  = \langle\alpha_0,I\otimes a^{\otimes i}\rangle  = \alpha_0(a,\cdots,a)(I) = \widetilde{\alpha_0}(\UL{a},a,\cdots,a)(I) = <\UL{a},a,\cdots,a|I> $$
where the equality in the middle follows from the identification
$$\mathrm{Hom}(A\otimes (A[1]/k\cdot 1)^{\otimes n},k)\cong \mathrm{Hom}((A[1]/k\cdot 1)^{\otimes n},A^*).$$
Hence, each term of the function $\rho$ of \cite{ATZ} equals the potential $\Psi$ in the paper given in the Lemma \ref{lem:nofrac}.
This proves the theorem.
\end{proof}

\bibliographystyle{amsalpha}

\begin{thebibliography}{FOOO}
\bibitem[ATZ]{ATZ} H. Abbaspour, T. Tradler, M. Zeinalian, {\em Algebraic string bracket as a Poisson bracket}
K-theory, 38 (2007), no.1, 59-82
\bibitem[C]{Cho} C.-H. Cho {\em Strong homotopy inner products of an $\AI$-algebra}, IMRN. ID 41. (2008),
\bibitem[CL]{CL} C.-H. Cho, S.-W. Lee {\em Notes on Kontsevich-Soibelman's theorem on cyclic $\AI$-algebras,} Imrn. (2010).
\bibitem[CO]{CO} C.-H. Cho, Y.-G. Oh
{\em Floer cohomology and disc instantons of Lagrangian torus fibers
in toric Fano manifolds}
Asian Journ. Math. 10 (2006), 773-814
\bibitem[Cos]{Cos} K. Costello
{\em Topological conformal field theories and Calabi-Yau categories.}
Adv. Math. 210 (2007), no. 1, 165-214.
\bibitem[Fu]{Fu} K. Fukaya,
{\em Counting pseudo-holomorphic discs in Calabi-Yau 3 fold}
\bibitem[FOOO]{FOOO} K. Fukaya, Y.-G. Oh, H. Ohta and K. Ono,
{\em Lagrangian intersection Floer theory-anomaly and
obstruction,} Part I, II,  AMS/IP Studies in Advanced Mathematics, 46.1. American Mathematical Society, Providence, RI; International Press, Somerville, MA, 2009.
\bibitem[FOOO1]{FOOO1} K. Fukaya, Y.-G. Oh, H. Ohta and K. Ono,
{\em Lagrangian Floer theory on toric manifolds,}
\bibitem[GJ]{GJ} E. Getzler and J. Jones {\em $\AI$-algebras and
the cyclic bar complex} Illinois Journal Math. 34 no. 2 (1990) 256-283
\bibitem[Kaj]{Kaj} H. Kajiura
{\em Noncommutative homotopy algebras associated with open strings,} Reviews in Mathematical Physics. 1 (2007), 1-99.
\bibitem[Ko]{K} M. Kontsevich
{\em Formal (non)commutative symplectic geometry } The Gelfand Mathematical seminars 1990-1992,  Birkhauser Boston (1993) 173-187
\bibitem[KS]{KS} M. Kontsevich and Y. Soibelman
{\em Notes on A-infinity algebras, A-infinity categories and non-commutative geometry. I }
Preprint, arXiv:math/0606241
\bibitem[T]{T} T. Tradler
{\em Infinity inner products}
J.Homotopy Relat. Struct. 3 (2008), no.1, 245-271.
\bibitem[TZ]{TZ} T. Tradler, M. Zeinalian
{\em Infinity structure of Poincare duality spaces}
 Algebr. Geom. Topol. 7 (2007), 233-260
\end{thebibliography}

\end{document}